\documentclass[reqno]{amsart}
\usepackage{amsthm}
\usepackage{enumerate}
\usepackage[centering]{geometry}
\usepackage{macro}

\newtheorem{thm}{Theorem}

\newtheorem{lemma}[thm]{Lemma}

\newcommand{\PCR}{Poincar\'e}
\newcommand{\tS}{\tilde{S}}
\newcommand{\ee}{\mathbf{e}}

\newcommand{\CH}{\Hess_{\CC}}
\newcommand{\RH}{\Hess}

\newcommand{\dH}{\tilde{H}}
\newcommand{\dS}{\tilde{S}}
\newcommand{\Herm}{\mathcal{H}}

\DeclareMathOperator{\Hess}{Hess}
\DeclareMathOperator{\dist}{dist}

\begin{document}
\title{Convexity of Hessian Integrals and Poincar\'e Type Inequalities}
\author{Zuoliang Hou}
\address{Mathematics Department, Columbia University, New York, NY 10027}
\email{hou@math.columbia.edu}
\date{\today}


\maketitle

\section{Introduction} 
Through out this paper, $k$ will be a fixed integer such that $0 \leq
k \leq n$.  For $\fl\in \RR^n$, let $S_k(\fl)$ be the normalized
$k^{\text{th}}$ elementary symmetric function of $\fl$, i.e. 
  $$ S_k(\fl_1,\cdots,\fl_n) = \frac{k!(n-k)!}{n!}
  \sum_{i_1<i_2<\cdots<i_k} \fl_{i_1}\cdots\fl_{i_k}. $$
Let $\Herm(n)$ denote the space of Hermitian matrices of size $n$.
For $A\in \Herm(n)$, denote $\fl(A)\in\RR^n$ the eigenvalue of $A$ and 
  $$ S_k(A) = S_k(\fl(A)).$$
Even though we use $S_k$ to denote functions defined on both $\RR^n$
and $\Herm(n)$, its meaning should be clear from the context. In
either case, we use $\dS_k$ to denote the complete polarization of
$S_k$. Let
  $$ \fG_k = \{ \fl \in \RR^n \,\mid\, S_j(\fl) \geq 0, j=1, \cdots, k \}$$
be the $k$-positive cone in $\RR^n$. 

Let $\fO$ be a smooth bounded
domain in $\CC^n$. For real valued function $u \in\cC^2(\fO)$, we
define the complex $k$-Hessian of $u$ by
  $$ H_k[u] = S_k \big(\CH(u) \big) $$
where $\CH(u)$ is the complex Hessian matrix of $u$. Let 
  $$ \cP_k(\fO)=\{ u \in \cC^2(\bar{\fO}) \,\mid\, \fl\big(\CH(u)\big)
  \in \fG_k \}$$ 
be the space of $k$-plurisubharmonic ($k$-psh) function, and
$\cP_k^0(\fO)$ be the subspace of $\cP_k(\fO)$ containing functions
vanishing on $\di\fO$. It is well known that for $v_j\in \cP_k(\fO),
j=1, \cdots, k$,
  $$ \dH_k[v_1,\cdots,v_k] = \dS_k\Big(\CH(v_1)\big), \cdots,
  \fl\big(\CH(v_k) \Big) \geq 0,$$
For 
$u, u_0,\cdots, u_k \in \cC^2(\bar{\fO})$, denote
  $$ F_k[u_0,\cdots, u_k] = \int_{\fO} (-u_0) \dH_k[u_1,\cdots,u_k] 
\qquad\text{and}\qquad I_k[u]=F_k[u,\cdots,u]=\int_{\fO} (-u) H_k[u].$$

\begin{thm} Notation as above
  \begin{enumerate}[(a)]
	\item 
	  For any $u_0,\cdots, u_k \in \cP^0_k(\fO)$, 
	  \begin{equation}
		F_k[u_0,\cdots,u_k] \leq \prod_{j=0}^k
		\big(I_k[u_j]\big)^{\frac{1}{k+1}}. 
		\label{eq:main}
	  \end{equation}
	\item For fixed $0 \leq m < k$, and fixed $v_1, \cdots, v_m
	  \in \cP_k(\fO)$, consider the functional
		$$ \cP_k^0(\fO) \ni u \mapsto I_{k,v_1,\cdots,v_m}[u] =
		\int_{\fO} (-u) \dH_k[u,\cdots, u, v_1, \cdots,v_m].$$
      Then $I_{k,v_1,\cdots,v_m}^{1/(k-m+1)}$ is convex on $\cP_0^k$ 
	  in the sense that for any $u, v \in \cP^0_k(\fO)$,
	  \begin{equation}
		I_{k,v_1,\cdots,v_m}^{\frac{1}{k-m+1}}[u+v] \leq
		I_{k,v_1,\cdots,v_m}^{\frac{1}{k-m+1}}[u] + 
		I_{k,v_1,\cdots,v_m}^{\frac{1}{k-m+1}}[v].
	  \end{equation}
	\end{enumerate}
  \label{thm:main}
\end{thm}

A smooth bounded domain $\fO$ is strong $k$-pseudoconvex, if $\di\fO$
is connected, and the $(k-1)^{\text{th}}$ elementary function of the
eigenvalues of the Levi form of $\di\fO$ is strictly positive on
$\di\fO$. For strong $k$-pseudoconvex domain, we have following
\PCR{} type inequalities.

\begin{thm} 
  Suppose $\fO$ is strong $k$-pseudoconvex, then for $0\leq m < k$,
  there exists constant $C$ such that for any
  $u\in \cP^0_k(\fO)$
  \begin{equation}
	\Big(I_m[u]\Big)^{\frac{1}{m+1}} \leq C \Big(I_k[u]\Big)^{\frac{1}{k+1}},
	\label{eq:cp}
  \end{equation}
  where the constant $C$ depends on $m$, $k$ and $\fO$ only.
  \label{thm:CPC}
\end{thm}

Some special cases of Theorem~\ref{thm:main} were well known. For
example Cegrell and Persson proved in \cite{CegrellPersson1997} that if
both $u$ and $v$ are psh function on a domain with vanishing boundary
data, then for integer $j$ such that $0\leq j\leq n$,
$$ \int_{\fO} (-u) (dd^c u)^j (dd^c v)^{n-j} \leq \Big(\int_{\fO} (-u)
(dd^c u)^n \Big)^{\frac{j+1}{n+1}} \Big(\int_{\fO} (-v) (dd^c v)^n
\Big)^{\frac{n-j}{n+1}}, $$
which is just special case of equation~(\ref{eq:main}) with
$u_0=\cdots=u_j=u$ and $u_{j+1}=\cdots=u_k=v$. 

Since the proof of Theorem~\ref{thm:main} only uses the divergence
structure of Hessian operator $H_k$, so it can also be applied to the
study of real Hessian operator and integrals.  Let $D$ be a smooth
bounded domain in $\RR^n$. For real valued function $u_j \in
\cC^2(\bar{D}), j=0,1,\cdots,k$ and $u\in\cC^2(\bar{\fO})$ define
$$ G_k[u_0,\cdots, u_k] = \int_D (-u_0)
\dS_k[\RH(u_1),\cdots,\RH(u_k)] $$
and
$$ J_k[u]= G_k[u,\cdots,u], $$
where $\RH(u)$ is the Hessian matrix of $u$. 

\begin{thm}
  For any $k$-convex function $u_0, \cdots, u_k$ with $u_0=\cdots=u_k=0$
  on $\di D$, 
  \begin{equation}
	G_k[u_0,\cdots,u_k] \leq \prod_{j=0}^n
	\Big(J_k[u_j]\Big)^{\frac{1}{k+1}}.
  \end{equation}
  \label{thm:realmain}
\end{thm}

\begin{thm}
  Suppose $D$ is smooth bounded domain in $\RR^n$ such that $S_{k-1}(II)>0$
  on $\di D$, where $II$ is the second fundamental form of $\di D$,
  then for integer $m$ with $0\leq m < k$, there exists constant $C$
  depends on $m$, $k$ and $D$ such that
  \begin{equation}
 	\Big(J_m[u]\Big)^{\frac{1}{m+1}} \leq C \Big(J_k[u]\Big)^{\frac{1}{k+1}} 
	\label{eq:RPC}
  \end{equation}
  for any $k$-convex function $u$ with $u=0$ on $\di D$. 
  \label{thm:RPC}
\end{thm}

Theorem~\ref{thm:RPC} was originally proved by Trudinger and Wang in
\cite{TrudingerWang1998} using parabolic method. Compared with the
proof in \cite{TrudingerWang1998}, our proof is quicker and more
elementary, and can be adopted to study the complex Hessian integral.
However, Trudinger and Wang's proof contains more information, they 
showed that the the best constant $C$ in equation~(\ref{eq:RPC}) is
attained by the solution of the Dirichlet problem: $H_k[u]=H_m[u]$ and
$u=0$ on $\di D$.

Hessian integrals like $J_k$ were first studied by
Bakelman~\cite{Bakelman1983} and Tso~\cite{Tso1990} in the setting of
real \MA-operator. In \cite{Wang1994}, Wang studied the real Hessian
integrals for general $k$ and established some Sobolev type
inequalities. For more information on real Hessian integral, see
\cite{Wang2009} and the references therein. 

For simplicity, we always assume that the functions studied in this
paper are at least twice differentiable. Using the standard
approximation technique, it is possible to extend the result to the
case of real and complex Hessian measure for non-smooth functions, see
\cite{TrudingerWang2002,Blocki2005} for more details.


\section{Complex Hessian Integrals}

Let $\fo$ is the \KH{} form of the standard Euclidean metric of $\CC^n$, i.e.
$$ \fo = \ii \sum_{j=1}^n dz^j \wedge d\bar{z}^j. $$
It is easy to check that for $u, u_0,\cdots, u_k \in \cP_k(\fO)$,
$$ H_k[u] \fo^n = (\ii\ddb u)^k \wedge \fo^{n-k} $$
and
$$\dH_k[u_1,\cdots,u_k] \fo^n = (\ii \ddb u_1)\wedge
\cdots \wedge (\ii \ddb u_k) \wedge \fo^{n-k}. $$
Hence
$$ I_k[u]=\int_{\fO} (-u) (\ii\ddb u)^k \wedge \fo^{n-k} $$
and
\begin{equation}
  F_k[u_0,\cdots,u_k] = \int_{\fO} (-u_0) (\ii \ddb u_1)\wedge \cdots
  \wedge (\ii \ddb u_k) \wedge \fo^{n-k}. 
  \label{eq:Fk}
\end{equation}
Obviously the $F_k$ in equation~\ref{eq:Fk} is symmetric in $u_1,
\cdots, u_k$.  If $u_j\in\cP^0_k(\fO), j=0,\cdots, k$, then by integration by part,
\begin{equation*}
  \begin{split}
	F_k[u_0,\cdots,u_k]&= \int_{\fO} \ii \di u_0 \wedge \dib u_1\wedge \cdots
	\wedge (\ii \ddb u_k) \wedge \fo^{n-k} \\ 
	&= \int_{\fO} (-u_1) (\ii \ddb u_0)\wedge \cdots
	\wedge (\ii \ddb u_k) \wedge \fo^{n-k} 
	= F_k[u_1,u_0, \cdots, u_k] 
  \end{split}
\end{equation*}
Therefore $F_k[u_0,\cdots,u_k]$ is symmetric in its arguments if
they are in $\cP^0_k(\fO)$.

\begin{lemma} For any $u_0,u_1\in \cP^0_k(\fO)$ and any $v_2,\cdots,v_k\in \cP_k(\fO)$, 
  \begin{equation}
	F_k[u_0,u_1,v_2,\cdots,v_k]^2 \leq F_k[u_0,u_0,v_2,\cdots,v_k]
	F_k[u_1,u_1,v_2,\cdots, v_k].
  \end{equation}
  \label{lemma:1}
\end{lemma}
\begin{proof}
  For $t\in \RR$, consider the quadratic function 
  \begin{equation*}
	\begin{split}
    q(t) &= F_k[u_0+tu_1,u_0+tu_1,v_2,\cdots,v_k] \\
	 &= F[u_1, u_1, v_2,\cdots, v_k]\; t^2 + 2 F[u_0,u_1,v_2,\cdots,v_k]\;t
	 + F[u_0, u_0, v_2, \cdots, v_k].
	\end{split}
  \end{equation*}
  Since
  \begin{equation*}
	\begin{split}
	  q(t)&=\int_{\fO} -(u_0+tu_1 ) (\ii \ddb u_0+tu_1) \wedge (\ii
	  \ddb v_2) \wedge (\ii ddb v_k) \wedge \fo^{n-k} \\
	  &= \int_{\fO} \ii \di (u_0+tu_1) \wedge \dib (u_0 + tu_1) \wedge (\ii\ddb v_2)
	  \wedge (\ii ddb v_k) \wedge \fo^{n-k} \\
	  &\geq 0.
	\end{split}
  \end{equation*}
  So
  $$ F_k[u_0,u_1,v_2,\cdots,v_k]^2 \leq F_k[u_0,u_0,v_2,\cdots,v_k]
	F_k[u_1,u_1,v_2,\cdots, v_k].$$
\end{proof}

With Lemma~\ref{lemma:1}, both statements in Theorem~\ref{thm:main} will follow from
following well known algebraic lemma.

\begin{lemma}
  Let $S$ be a set and $f$ be a non-negative symmetric function defined on
  $S^k=S \times \cdots \times S$.
  Assume that for any $x_1,\cdots, x_k \in S$, 
  $$ f(x_1,x_2,x_3,\cdots, x_k)^2 \leq f(x_1,x_1,x_3,\cdots,x_k)
  f(x_2,x_2,x_3,\cdots, x_k). $$
  Then it follows 
  $$ f(x_1,x_2,\cdots, x_k)^k \leq \prod_{j=1}^k
  f(x_j,\cdots,x_j).$$ 
  Moreover if $S$ is a cone in some linear space, and $f$ is linear in
  each of its argument, then
  $$ f(x+y,\cdots,x+y)^{1/k} \leq
  f(x,\cdots,x)^{1/k}+f(y,\cdots,y)^{1/k}$$
  for any $x$ and $y$ in $S$.
  \label{lemma:alg}
\end{lemma}

Proof of this lemma can be found in \cite{Hormander1994}.

To prove Theorem~\ref{thm:CPC}, we need another simple algebraic lemma for
symmetric functions.

\begin{lemma}
  Let $\mu \in \fG_k \subset \RR^n$ with $S_k(\mu)>0$, then there
  exists $C$ depending on $\dist(\mu,\di\fG_k)$ only such that for
  positive integer $m<k$ and $\forall\, \fl \in \fG_k$ 
  $$ S_m(\fl) \leq C \tS_k(\fl,\cdots,\fl, \mu,\cdots, \mu), $$
  where $\fl$ and $\mu$ appear $m$ and $k-m$ times in $\tS_k$
  respectively.
  \label{lemma:mk}
\end{lemma}
\begin{proof}
  Since $S_k(\mu)>0$, so $\mu$ is in the interior of $\fG_k$,
  therefore there exits $\vfe>0$ such that
  $$ \mu-\vfe \ee \in \fG_k. $$
  where $\ee=(1,\cdots,1)\in\RR^n$. Hence by the G{\aa}rding's Inequality
  \begin{equation}
	\begin{split}
      \tS_k^{\frac{1}{k-m}}(\fl,\cdots,\fl,\mu,\cdots,\mu) 
      & \geq \tS_k^{\frac{1}{k-m}}(\fl,\cdots,\fl,\mu-\vfe \ee,\cdots,\mu-\vfe \ee) +  
      \tS_k^{\frac{1}{k-m}}(\fl,\cdots,\fl,\vfe \ee,\cdots,\vfe \ee) \\
	  & \geq \tS_k^{\frac{1}{k-m}}(\fl,\cdots,\fl,\vfe \ee,\cdots,\vfe \ee)
	   = \vfe S_m^{\frac{1}{k-m}}(\fl).
	\end{split}
  \end{equation}
  Therefore
  $$ S_m(\fl) \leq \frac{1}{\vfe^{k-m}} \tS_k(\fl,\cdots,\fl,\mu,\cdots,\mu).$$
\end{proof}

\noindent
\emph{Proof of Theorem~\ref{thm:CPC}.} Since $\fO$ is strong
$k$-pseudoconvex, it is known \cite{Vinacua1988,Li2004} that there
exits a unique $v\in \cP_0^k(\fO)$ such that $H_k[v]=1$. Hence there
exists $\vfe>0$ depending on $v$ only, such that $\CH(v)-\vfe I_n \in
\fG_k \subset \Herm(n)$ pointwise in $\fO$, where $I_n$ is the
identity matrix of size $n$. Similar to the proof of
Lemma~\ref{lemma:mk}, there exists $C$ such that
$$ H_m[u]=S_m[\CH(u)] \leq C \dS_k[\CH(u),\cdots,\CH(u),
\CH(v),\cdots, \CH(v)] $$
where $\CH(u)$ and $\CH(v)$ appear $m$ and $k-m$ times in $\dS_k$.
Therefore
\begin{equation*}
  \begin{split}
	I_m[u] &= \int_{\fO} (-u) H_m[u] \\
	&\leq C \int_{\fO} (-u) \dS_k[\CH(u),\cdots,\CH(u), \CH(v),\cdots, \CH(v)]\\ 
	&= C F_k[u,\cdots,u,v,\cdots,v] \\
	&\leq C \big(I_k[u]\big)^{\frac{m+1}{k+1}} \big(I_k[v]\big)^{\frac{1}{k+1}}.
  \end{split}
\end{equation*}
So
$$ \big(I_m[u]\big)^{\frac{1}{m+1}} \leq \tilde{C} \big(I_k[u]\big)^{\frac{l+1}{k+1}}.$$
\qed

\section{Real Hessian Integrals}

To prove Theorem~\ref{thm:realmain} and Theorem~\ref{thm:RPC}, we only
need to show that the functional $G_k$ is symmetric in its arguments
when they vanish on the boundary and the analogous Lemma~\ref{lemma:1}
for $G_k$. They both follow from the divergence structure of
polarized real Hessian operator. 

Let $A^{(j)},j=1,2,\cdots,n$ be symmetric $n$ by $n$ matrices, then
the polarized symmetric function $\dS_k$ is given by 
$$ \dS_k(A^{(1)},\cdots, A^{(k)}) = \frac{1}{k!}
\sum_{\substack{i_1,\cdots,i_k\\ j_1,\cdots,j_k}} \fd^{i_1 \cdots
i_k}_{j_1\cdots j_k} A^{(1)}_{i_1 j_1} \cdots A^{(k)}_{i_k j_k}. $$
where $\fd^{i_1 \cdots i_k}_{j_1 \cdots j_k}$ is the generalized
Kronecker delta symbol. Denote
$$ \dS_{k-1}^{ij}(A^{(1)},\cdots, A^{(k-1)}) = \dod{\;}{A^{(k)}_{ij}}  
 \dS_k(A^{(1)},\cdots, A^{(k)}),$$
then
$$\dS_k(A^{(1)},\cdots, A^{(k)})=A^{(k)}_{ij}
\dS_{k-1}^{ij}(A^{(1)},\cdots, A^{(k-1)})$$
and
$$ \dS_{k-1}^{ij}(A^{(1)},\cdots, A^{(k-1)}) = \frac{1}{k!} 
\sum_{\substack{i_1,\cdots,i_{k-1}\\ j_1,\cdots,j_{k-1}}} \fd^{i_1 \cdots
i_{k-1} i}_{j_1\cdots j_{k-1} j} A^{(1)}_{i_1 j_1} \cdots A^{(k-1)}_{i_{k-1} j_{k-1}}. $$

\begin{lemma}
  Let $u_1,\cdots u_{k-1} \in \cC^2(D)$, then for any index
  $i\in\{1,\cdots, n\}$, 
  $$ \sum_{j=1}^n \dod{\;}{x_j}
  \dS^{ij}_{k-1}\Big(\RH(u_1),\cdots,\RH(u_{k-1})\Big)=0  $$
  \label{lemma:div}
\end{lemma}
\begin{proof} Denote $\Hess(u_m)=A^{(m)}, m=1,\cdots, k-1$, then
  \begin{equation}
	\begin{split}
	  \sum_j \dod{\dS^{ij}_{k-1}}{x_j} &=\frac{1}{k!}\sum_j\dod{\;}{x_j} 
	  \Big(\sum_{\substack{i_1,\cdots,i_{k-1}\\ j_1,\cdots,j_{k-1}}}
	  \fd^{i_1 \cdots i_{k-1} i}_{j_1\cdots j_{k-1} j} A^{(1)}_{i_1
	  j_1} \cdots A^{(k-1)}_{i_{k-1} j_{k-1}}\Big) \\
	  &= \frac{1}{k!}
	  \sum_{\substack{i_1,\cdots,i_{k-1}\\ j_1,\cdots,j_{k-1},j}}
	  \fd^{i_1 \cdots i_{k-1} i}_{j_1\cdots j_{k-1} j}
	  \sum_{m=1}^{k-1} A^{(1)}_{i_1 j_1} \cdots A^{(m-1)}_{i_{m-1} j_{m-1}}
	  A^{(m)}_{i_m j_m,j}A^{(m+1)}_{i_{m+1} j_{m+1}}\cdots A^{(k-1)}_{i_{k-1} j_{k-1}} 
	  \label{eq:kro}
	\end{split}
  \end{equation}
Noticed that the $j_m$ and $j$ are anti-symmetric in the Kronecker
symbol, while symmetric in $A^{m}_{i_m j_m, j}$, so the sum in
equation~(\ref{eq:kro}) must be zero.
\end{proof}

Once the Lemma~\ref{lemma:div} is proved, then for any $k$-convex
functions $v_2,\cdots, v_k$, and $\cC^2$ functions $u_0$ and $u_1$
with $u_0=u_1$ on $\di D$, 
\begin{equation*}
  \begin{split}
  G_k[u_0,u_1,v_2,\cdots,v_k] &= \int_{D} - u_0 \dS_k[
  (u_1)_{ij},(v_2)_{ij},\cdots,(v_k)_{ij}] \\
  &=\int_D - u_0 (u_1)_{ij} \dS^{ij}_{k-1} \\
  &=\int_D (u_0)_i (u_1)_j \dS^{ij}_{k-1}
\end{split}
\end{equation*}
So $G_k$ is symmetric in its arguments if they vanish on the
boundary, and more over 
$$ G_k[u,u,v_2,\cdots,v_k] = \int_D u_i u_j \dS^{ij}_{k-1} \geq 0 $$
for any $k$-convex $v_2,\cdots, v_k$, and $\cC^2$ function $u$ with
$u=0$ on $\di D$. Then same argument in the proof of
Lemma~\ref{lemma:1} can be used to show following lemma which implies
Theorem~\ref{thm:realmain} and Theorem~\ref{thm:RPC}.

\begin{lemma}
  For any $k$-convex functions $v_2,\cdots, v_k$, and any $\cC^2$
  functions $u_0$ and $u_1$. If $u_0=u_1=0$ on $\di D$, then
  $$ G_k[u_0, u_1,v_2,\cdots, v_k]^2 \leq G_k[u_0, u_0, v_2,\cdots,
  v_k] G_k[u_1,u_1,v_2, \cdots, v_k]. $$
\end{lemma}

\bibliographystyle{alpha}
\bibliography{poincare}
\end{document}